\documentclass[12pt,oneside,english,reqno]{amsart}
\usepackage[T1]{fontenc}
\usepackage[latin9]{inputenc}
\usepackage{babel}
\usepackage{varioref}
\usepackage{prettyref}
\usepackage{float}
\usepackage{amsthm}
\usepackage{amstext}
\usepackage{amssymb}
\usepackage{graphicx}
\usepackage{setspace}
\usepackage{esint}
\usepackage{xargs}[2008/03/08]
\usepackage[unicode=true,pdfusetitle,
 bookmarks=true,bookmarksnumbered=false,bookmarksopen=false,
 breaklinks=false,pdfborder={0 0 1},backref=false,colorlinks=false]
 {hyperref}
\usepackage{breakurl}

\makeatletter

\newcommand{\noun}[1]{\textsc{#1}}
\floatstyle{ruled}
\newfloat{algorithm}{tbp}{loa}
\providecommand{\algorithmname}{Algorithm}
\floatname{algorithm}{\protect\algorithmname}

\numberwithin{equation}{section}
\theoremstyle{plain}
\newtheorem{thm}{\protect\theoremname}
  \theoremstyle{definition}
  \newtheorem{defn}[thm]{\protect\definitionname}
  \theoremstyle{plain}
  \newtheorem{prop}[thm]{\protect\propositionname}
  \theoremstyle{plain}
  \newtheorem{conjecture}[thm]{\protect\conjecturename}
  \theoremstyle{plain}
  \newtheorem{lem}[thm]{\protect\lemmaname}
  \theoremstyle{remark}
  \newtheorem{rem}[thm]{\protect\remarkname}
  \theoremstyle{definition}
  \newtheorem{problem}[thm]{\protect\problemname}

\usepackage{dimadefs}

\DeclareMathOperator{\diag}{diag}

\newrefformat{cor}{Corollary \ref{#1}}
\newrefformat{def}{Definition \ref{#1}}
\newrefformat{rem}{Remark \ref{#1}}
\newrefformat{prop}{Proposition \ref{#1}}
\newrefformat{app}{Appendix \ref{#1}}
\newrefformat{alg}{Algorithm \ref{#1}}
\newrefformat{conj}{Conjecture \ref{#1}}
\newrefformat{prob}{Problem \ref{#1}}

\usepackage{fullpage}

\@ifundefined{showcaptionsetup}{}{%
 \PassOptionsToPackage{caption=false}{subfig}}
\usepackage{subfig}
\makeatother

  \providecommand{\conjecturename}{Conjecture}
  \providecommand{\definitionname}{Definition}
  \providecommand{\lemmaname}{Lemma}
  \providecommand{\problemname}{Problem}
  \providecommand{\propositionname}{Proposition}
  \providecommand{\remarkname}{Remark}
\providecommand{\theoremname}{Theorem}

\begin{document}

\title{Complete Algebraic Reconstruction of Piecewise-Smooth Functions from
Fourier Data}

\author{Dmitry Batenkov}

\address{Department of Mathematics\\
Weizmann Institute of Science\\
Rehovot 76100\\
Israel}

\email{dima.batenkov@weizmann.ac.il}

\urladdr{http://www.wisdom.weizmann.ac.il/~dmitryb}



\thanks{This research has been supported by the Adams Fellowship Program
of the Israel Academy of Sciences and Humanities.}

\date{\today}

\maketitle
\global\long\def\np{\ensuremath{K}}
\global\long\def\jp{\ensuremath{\xi}}
\global\long\def\w{\ensuremath{\omega}}
\global\long\def\ord{\ensuremath{d}}
\global\long\def\jc{\ensuremath{a}}
\newcommandx\fc[1][usedefault, addprefix=\global, 1=k]{\ensuremath{c_{#1}}}
\global\long\def\sc{\ensuremath{M}}
\global\long\def\fun{\ensuremath{f}}
\newcommandx\er[1][usedefault, addprefix=\global, 1=k]{\delta_{#1}}
\newcommandx\meas[1][usedefault, addprefix=\global, 1=k]{m_{#1}}
\newcommandx\apprmeas[1][usedefault, addprefix=\global, 1=k]{\widetilde{m}_{#1}}
\newcommandx\frsum[2][usedefault, addprefix=\global, 1=\fun, 2=\sc]{\mathfrak{F}_{#2}\left(#1\right)}
\global\long\def\smooth{\ensuremath{\Psi}}
\global\long\def\sing{\ensuremath{\Phi}}
\global\long\def\nn#1{\widetilde{#1}}
\global\long\def\e{\varepsilon}

\begin{abstract}
In this paper we provide a reconstruction algorithm for piecewise-smooth
functions with a-priori known smoothness and number of discontinuities,
from their Fourier coefficients, posessing the maximal possible asymptotic
rate of convergence -- including the positions of the discontinuities
and the pointwise values of the function. This algorithm is a modification
of our earlier method, which is in turn based on the algebraic method
of K.Eckhoff proposed in the 1990s. The key ingredient of the new
algorithm is to use a different set of Eckhoff's equations for reconstructing
the location of each discontinuity. Instead of consecutive Fourier
samples, we propose to use a ``decimated'' set which is evenly spread
throughout the spectrum.
\end{abstract}

\section{Introduction}

Consider the problem of reconstructing a function $\fun:\left[-\pi,\pi\right]\to\reals$
from a finite number of its Fourier coefficients
\[
\fc(\fun)\isdef\frac{1}{2\pi}\int_{-\pi}^{\pi}\fun(t)\ee^{-\imath kt}\dd t,\qquad k=0,1,\dots\sc.
\]

It is well-known that for periodic smooth functions, the truncated
Fourier series
\[
\frsum\isdef\sum_{|k|=0}^{\sc}\fc(\fun)\ee^{\imath kx}
\]
converges to $\fun$ very fast, subsequently making Fourier analysis
attractive for many applications. The precise dependence of the rate
of convergence on structural properties of $\fun$ is extensively
investigated in classical harmonic analysis and approximation theory
(see e.g. \cite{zygmund1959trigonometric}). In applications, it is
often sufficient to consider the number of continuous derivatives
of the function.
\begin{defn}
\label{def:cd}Let $C^{\ord+1}$ denote the class of continuous functions
having $\ord$ continuous derivatives, such that in addition $\der{\fun}{\ord+1}$
is piecewise-continuous and piecewise-differentiable.
\end{defn}
Applying integration by parts and the Riemann-Lebesgue lemma one has
immediately the following fact (see e.g. \cite[Section 3]{gottlieb1977numerical}).
\begin{prop}
For any $\fun\in C^{\ord+1}$ which is periodic (including its first
$\ord$ derivatives), we have $\left|\fc\left(\fun\right)\right|=O\left(\left|k\right|^{-\ord-2}\right)$,
while the approximation error is of the order
\begin{equation}
\left|\fun\left(x\right)-\frsum[\fun]\left(x\right)\right|=O\left(\sc^{-\ord-1}\right),\label{eq:best-approximation-smooth}
\end{equation}
and this holds uniformly in $\left[-\pi,\pi\right].$ 
\end{prop}
Yet many realistic phenomena exhibit discontinuities, in which case
the unknown function $\fun$ is only piecewise-smooth. As a result,
the trigonometric polynomial $\frsum$ no longer provides a good approximation
to $\fun$ due to the slow convergence of the Fourier series (one
of the manifestations of this fact is commonly known as the ``Gibbs
phenomenon''). It has very serious implications, for example when
using spectral methods to calculate solutions of PDEs with shocks
\cite{gottlieb1977numerical}.
\begin{defn}
\label{def:pcdk}Let $PC\left(\ord+1,\np\right)$ denote the class
of piecewise-smooth functions $\fun$ with $\np$ points of discontinuity
of the first kind, such that the restriction of $\fun$ on each continuity
interval is in $C^{\ord+1}$ (as in \prettyref{def:cd}).
\end{defn}
An important question arises: \emph{``Can such piecewise-smooth functions
be reconstructed from their Fourier measurements, with accuracy which
is comparable to the 'classical' one such as \eqref{eq:best-approximation-smooth}''?}

This problem has received much attention, especially in the last few
decades (\cite{banerjee1998exponentially,barkhudaryan2007asymptotic,marchbarone:2000,bauer-band,beckermann2008rgp,Boyd20091404,brezinski2004extrapolation,driscoll2001pba,eckhoff1995arf,eckhoff1998high,engelberg2008recovery,gelb1999detection,gottlieb1997gpa,guilpin2004eps,hrycak2010pseudospectral,kvernadze2004ajd,kvernadze2010approximation,mhaskar2000polynomial,Poghosyan2010,poghosyan2011auto,shizgal2003towards,solomonoff1995reconstruction,tadmor2007filters,wei2007detection}
would be only a partial list). It has long been known that the key
problem for Fourier series acceleration is the detection of the shock
locations. While efficient methods for edge detection exist (e.g.
concentration kernels of Tadmor et.al. \cite{engelberg2008recovery,gelb1999detection,tadmor2007filters}),
the theoretical analysis of these methods suggests that they provide
not more than first order accuracy. In contrast, our main interest
in this paper is to investigate achievability of the maximal theoretically
possible rate of convergence. Applying elementary considerations we
have the following fact (see proof in \prettyref{app:max-accuracy}).
\begin{prop}
\label{prop:maximal-accuracy}Let $\fun\in PC\left(\ord+1,\np\right)$.
Then no deterministic algorithm can restore the locations of the discontinuities
from $\left\{ \fc\left(\fun\right)\right\} _{\left|k\right|\leqslant\sc}$
with accuracy which is asymptotically higher than $\sc^{-\ord-2}$.
\end{prop}
Until now, the question of whether this maximal accuracy is achievable
remained open. During the 1990's, a certain method has been put forward
by K.Eckhoff in a series of papers \cite{eckhoff1993,eckhoff1995arf,eckhoff1998high},
which he conjectured to provide such accuracy (see \prettyref{sec:previous}).
Thus we have the following \emph{``Eckhoff's conjecture''.}
\begin{conjecture}[Eckhoff's conjecture]
\label{conj:eckhoff} The jump locations of a piecewise-smooth function
$\fun\in PC\left(\ord+1,\np\right)$ can be reconstructed from its
first $2\sc+1$ Fourier coefficients, with accuracy $O\left(\sc^{-\ord-2}\right)$,
by solving the perturbed nonlinear system of algebraic equations \eqref{eq:eckhoff-system}.
\end{conjecture}
In our previous work \cite{batyomAlgFourier} we have provided an
explicit reconstruction algorithm (\prettyref{alg:half-algo} \vpageref{alg:half-algo}),
based on original Eckhoff's method, which restored the jump locations
(and subsequently the pointwise values of the function between the
jumps) with ``half'' the maximal accuracy. In the present paper we
modify the method of \cite{batyomAlgFourier} (see \prettyref{alg:new-reconstruction-single-jump}
and \prettyref{alg:full-algo}) so that full asymptotic accuracy is
achieved (\prettyref{thm:full-accuracy-final}). The vital difference
of the new algorithm compared to the original Eckhoff's method (and
its modification from \cite{batyomAlgFourier}) is that when solving
the system \eqref{eq:eckhoff-system}, instead of consecutive Fourier
coefficients, we take ones that are evenly spaced throughout the whole
sampling range (thus we call the new method ``decimated Eckhoff's
algorithm'').

We describe the general approach, as well as our previous results
obtained in \cite{batyomAlgFourier}, in \prettyref{sec:previous}.
The modified algorithm is provided in \prettyref{sec:algorithm},
and its accuracy is analyzed in \prettyref{sec:accuracy}. Results
of some numerical simulations are presented in \prettyref{sec:numerics}.
We briefly discuss the optimality and some practical aspects of the
algebraic reconstruction algorithms in \prettyref{sec:optimality}.
Some possible extensions and generalizations are outlined in \prettyref{sec:extensions}.

\thanks{We would like to thank B.Adcock and Y.Yomdin for useful discussions.
We are also very grateful to the reviewer for many constructive suggestions.}

\section{\label{sec:previous}Eckhoff's method and half-order reconstruction}

Let us first briefly describe what has become known as the Eckhoff's
method (or Krylov-Gottlieb-Eckhoff method) for nonlinear Fourier reconstruction
of piecewise-smooth functions \cite{eckhoff1993,eckhoff1995arf,eckhoff1998high}.

Let $\fun\in PC\left(\ord+1,\np\right)$. Consequently, $\fun$ has
$\np>0$ jump discontinuities $\left\{ \jp_{j}\right\} _{j=1}^{\np}$
(they can be located also at $\pm\pi$, but not necessarily so). Furthermore,
in every segment $\left[\jp_{j-1},\jp_{j}\right]$ we have that $\fun\in C^{\ord+1}$.
Denote the associated jump magnitudes at $\jp_{j}$ by
\[
\jc_{\ell,j}\isdef\der{\fun}{\ell}(\jp_{j}^{+})-\der{\fun}{\ell}(\jp_{j}^{-}),\qquad\ell=0,1,\dots,\ord.
\]
We write the piecewise smooth $\fun$ as the sum $\fun=\smooth+\sing$,
where $\smooth\in C^{\ord+1}$ and $\sing(x)$ is a piecewise polynomial
of degree $\ord$, uniquely determined by $\left\{ \jp_{j}\right\} ,\left\{ \jc_{\ell,j}\right\} $
such that it ``absorbs'' all the discontinuities of $\fun$ and its
first $\ord$ derivatives. This idea is very old and goes back at
least to A.N.Krylov (\cite{barkhudaryan2007asymptotic,kantokryl62}).
Eckhoff derives the following explicit representation for $\sing(x)$:
\begin{equation}
\begin{split}\sing(x) & =\sum_{j=1}^{\np}\sum_{\ell=0}^{\ord}\jc_{\ell,j}V_{\ell}(x;\jp_{j})\\
V_{n}\left(x;\jp_{j}\right) & =-\frac{\left(2\pi\right)^{n}}{\left(n+1\right)!}B_{n+1}\left(\frac{x-\jp_{j}}{2\pi}\right)\qquad\jp_{j}\leq x\leq\jp_{j}+2\pi
\end{split}
\label{eq:sing-part-explicit-bernoulli}
\end{equation}
where $V_{n}\left(x;\jp_{j}\right)$ is understood to be periodically
extended to $\left[-\pi,\pi\right]$ and $B_{n}(x)$ is the $n$-th
Bernoulli polynomial. Elementary integration by parts gives the following
formula.
\begin{prop}
Let $\sing(x)$ be given by \eqref{eq:sing-part-explicit-bernoulli}.
Then
\begin{equation}
\fc(\sing)=\frac{1}{2\pi}\sum_{j=1}^{\np}\ee^{-\imath k\jp_{j}}\sum_{\ell=0}^{\ord}(\imath k)^{-\ell-1}\jc_{\ell,j}.\label{eq:singular-fourier-explicit}
\end{equation}

\end{prop}
Eckhoff observed that if $\smooth$ is sufficiently smooth, then the
contribution of $\fc(\smooth)$ to $\fc(\fun)$ is negligible \textbf{for
large} $k$, and therefore one can hope to reconstruct the unknown
parameters $\left\{ \jp_{j},\jc_{\ell,j}\right\} $ from the perturbed
equations

\begin{equation}
\fc\left(\fun\right)=\frac{1}{2\pi}\sum_{j=1}^{\np}\ee^{-\imath k\jp_{j}}\sum_{\ell=0}^{\ord}(\imath k)^{-\ell-1}\jc_{\ell,j}+O\left(k^{-\ord-2}\right),\; k\gg1.\label{eq:eckhoff-system}
\end{equation}

His proposed method was to construct from the known values
\[
\left\{ \fc\left(\fun\right)\right\} \qquad k=\sc-\left(\ord+1\right)\np+1,\sc-\left(\ord+1\right)\np+2,\dots,\sc
\]
a system of algebraic equations satisfied by the jump points $\left\{ \jp_{1},\dots,\jp_{\np}\right\} $,
and solve this system numerically. Based on some explicit computations
for small values of $\ord,\np$ and large number of numerical experiments,
he conjectured that his method would reconstruct the jump locations
with accuracy $\sc^{-\ord-2}$ (\prettyref{conj:eckhoff}). 

In \cite{batyomAlgFourier} we proposed a reconstruction method based
on the original Eckhoff's procedure, outlined in \prettyref{alg:half-algo}
\vpageref{alg:half-algo}.

\begin{algorithm}[H]
\begin{raggedright}
Let $f\in PC\left(\ord+1,\np\right)$, and assume that $\fun=\sing^{\left(\ord\right)}+\smooth$
where $\sing^{\left(\ord\right)}$ is the piecewise polynomial absorbing
all discontinuities of $\fun$, and $\smooth\in C^{\ord+1}.$ Assume
in addition the following a-priori bounds:
\par\end{raggedright}
\begin{itemize}
\item Minimal separation distance between the jumps
\[
\min_{i\neq j}\left|\jp_{i}-\jp_{j}\right|\geqslant J>0;
\]

\item Upper bound on jump magnitudes
\[
\left|\jc_{l,j}\right|\leqslant A<\infty;
\]

\item Lower bound on the value of the lowest-order jump
\[
\left|\jc_{0,j}\right|\geqslant B>0;
\]

\item Upper bound on the size of the Fourier coefficients of $\smooth$:
\[
\left|\fc\left(\smooth\right)\right|\leqslant R\cdot k^{-\ord-2}.
\]

\end{itemize}
\begin{raggedright}
Let us be given the first $\sc\gg1$ Fourier coefficients of $\fun$
for $\sc>\sc\left(\ord,\np,J,A,B,R\right)$ (a quantity which is computable).
The reconstruction is as follows.
\par\end{raggedright}
\begin{enumerate}
\item Obtain first-order approximations to the jump locations $\left\{ \jp_{1},\dots,\jp_{\np}\right\} $
by Prony's method (Eckhoff's method of order 0).
\item Localize each discontinuity $\jp_{j}$ by calculating the first $\sc$
Fourier coefficients of the function $f_{j}=\fun\cdot h_{j}$ where
$h_{j}$ is a $C^{\infty}$ bump function satisfying

\begin{enumerate}
\item $h_{j}\equiv0$ on the complement of $\left[\jp_{j}-J,\jp_{j}+J\right]$;
\item $h_{j}\equiv1$ on $\left[\jp_{j}-\frac{J}{3},\jp_{j}+\frac{J}{3}\right]$.
\end{enumerate}
\item Fix the reconstruction order $\ord_{1}\leq\left\lfloor \frac{\ord}{2}\right\rfloor $.
For each $j=1,2,\dots,\np$, recover the parameters $\left\{ \jp_{j},\jc_{0,j},\dots,\jc_{\ord_{1},j}\right\} $
from the approximate system of $\ord_{1}+2$ equations
\begin{equation}
\fc\left(f_{j}\right)=\frac{1}{2\pi}\ee^{-\imath\jp_{j}k}\sum_{\ell=0}^{\ord_{1}}\frac{\jc_{\ell,j}}{\left(\imath k\right)^{\ell+1}}+\er,\qquad k=\sc-\ord_{1}-1,\sc-\ord_{1},\dots,\sc,\label{eq:half-algo-single-jump}
\end{equation}
by Eckhoff's method for one jump. The actual method is to solve a
single polynomial equation of degree $\ord_{1}$ constructed from
the measurements $\left\{ \fc\left(\fun_{j}\right)\right\} $, thus
recovering the unknown $\jp_{j}$, and subsequently solve a linear
system w.r.t. the rest of the parameters $\left\{ \jc_{0,j},\dots,\jc_{\ord_{1},j}\right\} $.
\item From the previous steps we obtained approximate values for the parameters
$\left\{ \nn{\jp_{j}}\right\} $ and $\left\{ \nn{\jc}_{\ell,j}\right\} $.
The final approximation is taken to be
\begin{equation}
\begin{split}\widetilde{\fun} & =\widetilde{\smooth}+\widetilde{\sing}=\sum_{\left|k\right|\leq\sc}\left\{ \fc(\fun)-\frac{1}{2\pi}\sum_{j=1}^{\np}\ee^{-\imath\nn{\jp_{j}}k}\sum_{\ell=0}^{d_{1}}\frac{\widetilde{\jc}_{\ell,j}}{(\imath k)^{\ell+1}}\right\} \ee^{\imath kx}+\sum_{j=1}^{\np}\sum_{\ell=0}^{d_{1}}\nn{\jc}_{\ell,j}V_{\ell}(x;\nn{\jp_{j}}).\end{split}
\label{eq:final-approximation}
\end{equation}

\end{enumerate}
\caption{Half-order algorithm, \cite{batyomAlgFourier}.}
\label{alg:half-algo}

\end{algorithm}

We have also shown that this method achieves the following accuracy.
\begin{thm}[\cite{batyomAlgFourier}]
\label{thm:half-accuracy}Let $\fun\in PC\left(\ord+1,\np\right)$
and let $\nn{\fun}$ be the approximation of order $\ord_{1}\leqslant\left\lfloor \frac{\ord}{2}\right\rfloor $
computed by \prettyref{alg:half-algo}. Then for large enough $\sc$
we have%
\footnote{The last (pointwise) bound holds on ``jump-free'' regions.%
}
\begin{equation}
\begin{split}\left|\nn{\jp_{j}}-\jp_{j}\right| & \leqslant C_{1}\left(\ord,\ord_{1},\np,J,A,B,R\right)\cdot\sc^{-\ord_{1}-2};\\
\left|\nn{\jc}_{\ell,j}-\jc_{\ell,j}\right| & \leqslant C_{2}\left(\ord,\ord_{1},\np,J,A,B,R\right)\cdot\sc^{\ell-\ord_{1}-1},\;\ell=0,1,\dots,\ord_{1};\\
\left|\nn{\fun}\left(x\right)-\fun\left(x\right)\right| & \leqslant C_{3}\left(\ord,\ord_{1},\np,J,A,B,R\right)\cdot\sc^{-\ord_{1}-1}.
\end{split}
\label{eq:half-estimates}
\end{equation}

\end{thm}
The non-trivial part of the proof of this result was to analyze in
detail the polynomial equation $p\left(\jp_{j}\right)=0$ in step
3 of \prettyref{alg:half-algo}. It turned out that additional orders
of smoothness (namely, between $\ord_{1}$ and $\ord$) produced an
error term $\er$ in \eqref{eq:half-algo-single-jump} which, when
substituted into the polynomial $p$, resulted in unexpected cancellations
due to which the root $\jp_{j}$ was perturbed only by $O\left(\sc^{-\ord_{1}-2}\right)$.
This phenomenon was first noticed by Eckhoff himself in \cite{eckhoff1995arf}
for $\ord=1$, but at the time its full significance was not realized.

\section{\label{sec:algorithm}The decimated Eckhoff algorithm}

\global\long\def\jcc{\alpha}
\global\long\def\scc{N}

In this section we present the ``decimated Eckhoff algorithm'',
which has a single essential (but crucial) difference compared to
\prettyref{alg:half-algo}. The difference is that in step 3, we solve
the ``full-order'' system, while choosing the indices $k$ to be
evenly distributed across the range $\left\{ 0,1,\dots,\sc\right\} $
(instead of the original choice $k=\sc-\ord-1,\sc-\ord,\dots,\sc$).
That is, denoting 
\[
\scc\isdef\left\lfloor \frac{\sc}{\left(\ord+2\right)}\right\rfloor ,
\]
the modified system \eqref{eq:half-algo-single-jump} reads

\begin{equation}
\nn{\fc}=\underbrace{\frac{\w^{k}}{2\pi}\sum_{\ell=0}^{\ord}\frac{\jc_{\ell}}{\left(\imath k\right)^{\ell+1}}}_{\isdef\fc}+\epsilon_{k},\quad k=\scc,2\scc,\dots,\left(\ord+2\right)\scc,\quad\left|\epsilon_{k}\right|\leq R\cdot k^{-\ord-2}.\label{eq:decimated-system-single-jump}
\end{equation}

Here $\w=\ee^{-\imath\jp}$ with $\jp=\jp_{j}\in\left[-\pi,\pi\right]$
being the unknown location of the (single) discontinuity of the localized
function $\fun_{j}$ (see step 2 of \prettyref{alg:half-algo}). 

The decimated system \eqref{eq:decimated-system-single-jump} is solved
in two steps. First, a polynomial equation $q_{\scc}^{\ord}\left(u\right)=0$
is constructed from the values $\left\{ \nn{\fc}\right\} _{k=\scc,2\scc,\dots,\left(\ord+2\right)\scc}$.
This $q_{\scc}^{\ord}$ is in fact a perturbation of an ``exact''
equation $p_{\scc}^{\ord}\left(u\right)=0$, constructed from the
unperturbed (and unknown) values $\left\{ \fc\right\} $ as in \eqref{eq:decimated-system-single-jump}.
This $p_{\scc}^{\ord}$ is defined explicitly below in \eqref{eq:pnd-def}.
As we show in \prettyref{prop:pnd-exact-root-z}, one of the roots
of this exact equation is the value $z=\w^{\scc}$. Thus, by solving
the perturbed equation $q_{\scc}^{\ord}\left(u\right)=0$ we recover
the unknown $\nn z=\ee^{-\imath\nn{\jp}\scc}$, and by extracting
$\scc^{\text{th}}$ root and subsequently taking logarithm we obtain
the approximation to the jump $\nn{\jp}$. The operation of taking
root generally results in a multi-valued solution%
\footnote{For example, if $\scc=2$ then the solution $z=1$ corresponds to
either $\jp=0$ or $\jp=\pm\pi$. In the general case, there are $\scc$
possible solutions, as follows: 
\begin{eqnarray*}
\ee^{\imath\jp\scc} & = & \ee^{\imath t}\\
\jp\scc-t & = & 2\pi n\\
\jp & = & \frac{t}{\scc}+\frac{2\pi}{\scc}n,\qquad n\in\integers.
\end{eqnarray*}
}. Therefore, to ensure correct reconstruction, we need an additional
assumption that the jump $\jp$ must be known with \emph{a-priori}
accuracy of the order $o\left(\scc^{-1}\right)$. Once the approximate
jump location $\nn{\jp}$ is reconstructed, the jump magnitudes $\left\{ \jc_{\ell,j}\right\} _{\ell=0}^{\ord}$
are recovered by solving a linear system of equations \eqref{eq:the-linear-system}.

The above procedure for recovery of a single jump  is summarized in
\prettyref{alg:new-reconstruction-single-jump} \vpageref{alg:new-reconstruction-single-jump}.
The complete algorithm is outlined in \prettyref{alg:full-algo} \vpageref{alg:full-algo}.

Let us now define the ``exact'' equation $p_{\scc}^{\ord}\left(u\right)=0$.
Denote $\jcc_{\ell}=\imath^{\ord+1-\ell}\jc_{\ord-\ell}$ and let
\begin{equation}
\meas\isdef\w^{k}\sum_{\ell=0}^{\ord}\jcc_{\ell}k^{\ell}.\label{eq:mk-def}
\end{equation}

With this notation, multiply both sides of \eqref{eq:decimated-system-single-jump}
by $\left(2\pi\right)\left(\imath k\right)^{\ord+1}$ and get
\begin{equation}
\apprmeas\isdef2\pi\left(\imath k\right)^{\ord+1}\widetilde{\fc}=\meas+\er,\qquad k=\scc,2\scc,\dots,\left(\ord+2\right)\scc,\quad\left|\er\right|\leqslant R\cdot k^{-1}.\label{eq:noisy-measurements}
\end{equation}

Recall that we have defined $z=\w^{\scc}$. Therefore we have by \eqref{eq:mk-def}
\begin{equation}
\meas[\left(j+1\right)\scc]=z^{j+1}\sum_{\ell=0}^{\ord}\jcc_{\ell}\left(j+1\right)^{\ell}\scc^{\ell}.\label{eq:mjN}
\end{equation}

\begin{defn}
Let
\begin{equation}
p_{\scc}^{\ord}\left(u\right)\isdef\sum_{j=0}^{\ord+1}\left(-1\right)^{j}{\ord+1 \choose j}\meas[\left(j+1\right)\scc]u^{\ord+1-j}.\label{eq:pnd-def}
\end{equation}
\end{defn}
\begin{prop}
\label{prop:pnd-exact-root-z}The point $u=z$ is a root of $p_{\scc}^{\ord}\left(u\right)$.\end{prop}
\begin{proof}
From \eqref{eq:mjN} and \eqref{eq:pnd-def} we have
\begin{eqnarray*}
p_{\scc}^{\ord}\left(z\right) & = & \sum_{j=0}^{\ord+1}\left(-1\right)^{j}{\ord+1 \choose j}z^{j+1}\sum_{\ell=0}^{\ord}\jcc_{\ell}\left(j+1\right)^{\ell}\scc^{\ell}z^{\ord+1-j}\\
 & = & z^{\ord+2}\sum_{\ell=0}^{\ord}\jcc_{\ell}\scc^{\ell}\left\{ \sum_{j=0}^{\ord+1}\left(-1\right)^{j}{\ord+1 \choose j}\left(j+1\right)^{\ell}\right\} .
\end{eqnarray*}
The expression in the curly braces is just the $\ord+1$-st forward
difference operator applied to the polynomial function $\varphi\left(k\right)=k^{\ell}$.
Since $\ell<\ord+1$, this is always zero (see e.g. \cite{elaydi2005ide}). 
\end{proof}
Now let us explicitly write the linear sysem for the jump magnitudes.
\begin{defn}
Let $V_{\scc}^{\ord}$ denote the $\left(\ord+1\right)\times\left(\ord+1\right)$
matrix
\[
V_{\scc}^{\ord}\isdef\begin{bmatrix}1 & \scc & \scc^{2} & \dots & \scc^{\ord}\\
1 & 2\scc & \left(2\scc\right)^{2} & \dots & \left(2\scc\right)^{\ord}\\
\vdots & \vdots & \vdots & \vdots & \vdots\\
1 & \left(\ord+1\right)\scc & \left(\left(\ord+1\right)\scc\right)^{2} & \dots & \left(\left(\ord+1\right)\scc\right)^{\ord}
\end{bmatrix}.
\]

\end{defn}
Note that $V_{\scc}^{\ord}$ is the Vandermonde matrix on the points
$\left\{ \scc,2\scc,\dots,\left(\ord+1\right)\scc\right\} $ and thus
it is non-degenerate for all $\scc\geqslant1$.
\begin{prop}
The vector of exact magnitudes $\left\{ \jcc_{j}\right\} $ satisfies
\begin{equation}
\begin{bmatrix}\meas[\scc]\w^{-\scc}\\
\meas[2\scc]\w^{-2\scc}\\
\vdots\\
\meas[\left(\ord+1\right)\scc]\w^{-\left(\ord+1\right)\scc}
\end{bmatrix}=V_{\scc}^{\ord}\begin{bmatrix}\jcc_{0}\\
\jcc_{1}\\
\vdots\\
\jcc_{\ord}
\end{bmatrix}.\label{eq:the-linear-system}
\end{equation}
\end{prop}
\begin{proof}
Immediately follows from \eqref{eq:mk-def}.
\end{proof}
\begin{algorithm}[H]
\begin{raggedright}
Let there be given the first $\scc\left(\ord+2\right)\gg1$ Fourier
coefficients of the function $\fun_{j}$ as in \eqref{eq:decimated-system-single-jump},
and assume that the jump position $\jp$ is already known with accuracy
$o\left(\scc^{-1}\right)$.
\par\end{raggedright}
\begin{enumerate}
\item Construct the polynomial 
\begin{eqnarray*}
q_{\scc}^{\ord}\left(u\right) & = & \sum_{j=0}^{\ord+1}\left(-1\right)^{j}{\ord+1 \choose j}\apprmeas[\left(j+1\right)\scc]u^{\ord+1-j}
\end{eqnarray*}
 from the given perturbed measurements $\apprmeas[\scc],\apprmeas[2\scc],\dots,\apprmeas[\left(\ord+2\right)\scc]$
as in \eqref{eq:noisy-measurements}.
\item Find the root $\widetilde{z}$ which is closest to the unit circle
(in fact any root will suffice, see \prettyref{rem:all-roots-are-good}
below).
\item Take $\widetilde{\w}=\sqrt[\scc]{\widetilde{z}}$. Note that in general
there are $\scc$ possible values on the unit circle, but since we
already know the approximate location of $\w$, the correct value
can be chosen consistently.
\item Set $\widetilde{\jp}=-\arg\widetilde{\omega}$.
\item To recover the magnitudes, solve the perturbed linear system \eqref{eq:the-linear-system}:
\begin{equation}
\begin{bmatrix}\apprmeas[\scc]\widetilde{\w}_{\scc}^{-\scc}\\
\apprmeas[2\scc]\widetilde{\w}_{\scc}^{-2\scc}\\
\vdots\\
\apprmeas[\left(\ord+1\right)\scc]\widetilde{\w}_{\scc}^{-\left(\ord+1\right)\scc}
\end{bmatrix}=V_{\scc}^{\ord}\begin{bmatrix}\widetilde{\jcc}_{0}\\
\widetilde{\jcc}_{1}\\
\vdots\\
\widetilde{\jcc}_{\ord}
\end{bmatrix}.\label{eq:linear-system-perturbed}
\end{equation}

\end{enumerate}
\caption{Recovery of single jump parameters}
\label{alg:new-reconstruction-single-jump}
\end{algorithm}

\begin{algorithm}[H]
\begin{raggedright}
Let $f\in PC\left(\ord+1,\np\right)$, and assume that $\fun=\sing^{\left(\ord\right)}+\smooth$
where $\sing^{\left(\ord\right)}$ is the piecewise polynomial absorbing
all discontinuities of $\fun$, and $\smooth\in C^{\ord+1}.$ Assume
the a-priori bounds as in \prettyref{alg:half-algo}.
\par\end{raggedright}
\begin{enumerate}
\item Using \prettyref{alg:half-algo}, obtain approximate values of the
jumps (up to accuracy $O\left(\scc^{-\left\lfloor \frac{\ord}{2}\right\rfloor -2}\right)$)
and the Fourier coefficients of the functions $\fun_{j}$. (By \prettyref{thm:half-accuracy}
this is indeed possible.)
\item Use \prettyref{alg:new-reconstruction-single-jump} to further improve
the accuracy of reconstructing the jumps $\left\{ \nn{\jp}_{j}\right\} _{j=1}^{\np}$
and the magnitudes $\left\{ \nn{\jc}_{\ell,j}\right\} $.
\item Take the final approximation as defined in \eqref{eq:final-approximation}.
\end{enumerate}
\caption{Full accuracy Fourier approximation}
\label{alg:full-algo}
\end{algorithm}

\section{\label{sec:accuracy}Main result}

The key result of this paper is the following.
\begin{thm}
\label{thm:jump-accuracy-new}Assume that
\begin{equation}
\left|\jc_{\ell}\right|\leqslant A^{*}<\infty,\quad\left|\jc_{0}\right|\geqslant B^{*}>0.\label{eq:single-jump-apriori-bounds}
\end{equation}
Then for $\scc\gg1$, \prettyref{alg:new-reconstruction-single-jump}
recovers the parameters of a single jump from the data $\apprmeas$
(given by \eqref{eq:noisy-measurements}) with the following accuracy:
\begin{eqnarray*}
\left|\nn{\jp}-\jp\right| & \leqslant & C_{4}\frac{R^{*}}{B^{*}}\scc^{-\ord-2},\\
\left|\nn{\jcc}_{\ell}-\jcc_{\ell}\right| & \leqslant & C_{5,\ell}R^{*}\left(1+\frac{A^{*}}{B^{*}}\right)\scc^{-\ell-1},\qquad\ell=0,1,\dots,\ord,
\end{eqnarray*}
where $R^{*}$ is some constant for which $ $\eqref{eq:noisy-measurements}
holds, $C_{4}$ depends only on $\ord$ and $C_{5,\ell}$ depends
only on $\ell$ and $\ord$.
\end{thm}
An immediate consequence is the resolution of Eckhoff's conjecture.
\begin{thm}
\label{thm:full-accuracy-final}Let $\fun\in PC\left(\ord+1,\np\right)$
and let $\nn{\fun}$ be the approximation of order $\ord$ computed
by \prettyref{alg:full-algo}. Then for $\sc\gg1$
\begin{equation}
\begin{split}\left|\nn{\jp_{j}}-\jp_{j}\right| & \leqslant C_{6}\left(\ord,\np,J,A,B,R\right)\cdot\sc^{-\ord-2};\\
\left|\nn{\jc}_{\ell,j}-\jc_{\ell,j}\right| & \leqslant C_{7}\left(\ell,\ord,\np,J,A,B,R\right)\cdot\sc^{\ell-\ord-1},\qquad\ell=0,1,\dots,\ord;\\
\left|\nn{\fun}\left(x\right)-\fun\left(x\right)\right| & \leqslant C_{8}\left(\ord,\np,J,A,B,R\right)\cdot\sc^{-\ord-1}.
\end{split}
\label{eq:full-estimates}
\end{equation}
\end{thm}
\begin{proof}
By Theorem 5.2 of \cite{batyomAlgFourier}, the Fourier coefficients
of the localized functions $f_{j}$ have error bounded by $R^{'}k^{-\ord-2}$
where the constant $R^{'}$ depends in general on all the a-priori
bounds, but not on $\sc$. Therefore the a-priori bounds required
by \prettyref{thm:jump-accuracy-new} are satisfied by $R^{*}=R',\; A^{*}=A$
and $B^{*}=B$. Therefore, the estimates of \prettyref{thm:jump-accuracy-new}
hold for each discontinuity $j=1,\dots,\np$. After substituting $\sc=\left(\ord+2\right)\scc$
and $\jcc_{\ell}=\imath^{\ord+1-\ell}\jc_{\ord-\ell,j}$, we get the
first two lines of \eqref{eq:full-estimates}. To get the pointwise
estimate $\left|\nn{\fun}\left(x\right)-\fun\left(x\right)\right|$,
just repeat the proof of Theorem 6.1 of \cite{batyomAlgFourier} \emph{verbatim}. 
\end{proof}
The remainder of this section is devoted to proving \prettyref{thm:jump-accuracy-new}.

Let us first define an auxiliary polynomial sequence.
\begin{defn}
For all $i,\ord$ nonnegative integers let
\[
s_{i}^{\ord}\left(w\right)\isdef\sum_{j=0}^{\ord+1}\left(-1\right)^{j}{\ord+1 \choose j}\left(j+1\right)^{i}w^{\ord+1-j}.
\]
\end{defn}
\begin{prop}
Let $w=\frac{u}{z}$ (recall that $z=\w^{\scc}$). Then 
\begin{equation}
p_{\scc}^{\ord}\left(u\right)=z^{\ord+2}\sum_{i=0}^{\ord}\jcc_{i}\scc^{i}s_{i}^{\ord}\left(w\right).\label{eq:p-decomposition}
\end{equation}
\end{prop}
\begin{proof}
By \eqref{eq:mjN} and \eqref{eq:pnd-def} we have 
\begin{eqnarray*}
p_{\scc}^{\ord}\left(zw\right) & = & \sum_{j=0}^{\ord+1}\left(-1\right)^{j}{\ord+1 \choose j}z^{j+1}\sum_{i=0}^{\ord}\jcc_{i}\left(j+1\right)^{i}\scc^{i}\left(zw\right)^{\ord+1-j}\\
 & = & z^{\ord+2}\sum_{i=0}^{\ord}\jcc_{i}\scc^{i}\sum_{j=0}^{\ord+1}\left(-1\right)^{j}{\ord+1 \choose j}\left(j+1\right)^{i}w^{\ord+1-j}\\
 & = & z^{\ord+2}\sum_{i=0}^{\ord}\jcc_{i}\scc^{i}s_{i}^{\ord}\left(w\right).
\end{eqnarray*}

\end{proof}
The most immediate conclusion of the formula \eqref{eq:p-decomposition}
is that the asymptotic properties of the polynomials $p_{\scc}^{\ord}$
are eventually determined by the corresponding properties of the fixed
polynomial $s_{\ord}^{\ord}$.
\begin{lem}
\label{lem:simple-roots}The polynomial $s_{\ord}^{\ord}\left(w\right)$
is square-free, and all of its roots belong to the interval $[1,+\infty)$.\end{lem}
\begin{proof}
We divide the proof into several steps.
\begin{enumerate}
\item First, notice that we have the following recursion:
\begin{equation}
s_{i+1}^{\ord}\left(w\right)=\left(\ord+2\right)s_{i}^{\ord}\left(w\right)-w\frac{\dd}{\dd w}s{}_{i}^{\ord}\left(w\right).\label{eq:sid-recursion}
\end{equation}

Indeed,{\small{
\begin{eqnarray*}
\left(\ord+2\right)s_{i}^{\ord}\left(w\right)-w\frac{\dd}{\dd w}s{}_{i}^{\ord}\left(w\right) & = & \left(\ord+2\right)\sum_{j=0}^{\ord+1}\left(-1\right)^{j}{\ord+1 \choose j}\left(j+1\right)^{i}w^{\ord+1-j}\\
 &  & -w\sum_{j=0}^{\ord}\left(-1\right)^{j}\left(\ord+1-j\right){\ord+1 \choose j}\left(j+1\right)^{i}w^{\ord+1-j}\\
 & = & \sum_{j=0}^{\ord+1}\left(-1\right)^{j}{\ord+1 \choose j}\left(j+1\right)^{i}w^{\ord+1-j}\left\{ \ord+2-\left(\ord+1-j\right)\right\} \\
 & = & \sum_{j=0}^{\ord+1}\left(-1\right)^{j}{\ord+1 \choose j}\left(j+1\right)^{i+1}w^{\ord+1-j}\\
 & = & s_{i+1}^{\ord}\left(w\right).
\end{eqnarray*}
}}{\small \par}

\item Next, notice that
\begin{equation}
s_{i+1}^{\ord}\left(w\right)=w^{\ord+3}\frac{\dd}{\dd w}\left[\frac{1}{w^{\ord+2}}s_{i}^{\ord}\left(w\right)\right].\label{eq:sid-diff-rel}
\end{equation}

\item By Rolle's theorem applied to \eqref{eq:sid-diff-rel}, we obtain
that there is a root of $s_{i+1}^{\ord}$ between any two consecutive
roots of $s_{i}^{\ord}$. 
\item Direct computation gives for $i=1$
\[
s_{1}^{\ord}\left(w\right)=\left(w-1\right)^{\ord}\left(w-\left(\ord+2\right)\right),
\]
and therefore the biggest root of $s_{1}^{\ord}\left(w\right)$ is
simple. Let us show by induction that this property is preserved for
all $i\geqslant1$. Let $y_{i}>0$ be the biggest root of $s_{i}^{\ord}$,
which is by assumption simple. Since the leading coefficient of $s_{i}^{\ord}$
is positive, we must have that $\frac{\dd}{\dd w}s_{i}^{\ord}\left(w\right)\big|_{w=y_{i}}>0$.
Therefore, by \eqref{eq:sid-recursion} we get $s_{i+1}^{\ord}\left(y_{i}\right)<0$,
so there must be a root of $s_{i+1}^{\ord}$ bigger than $y_{i}$.
By counting roots and using item (3), this new root must be simple.
\item Starting with $s_{0}^{\ord}\left(w\right)=\left(w-1\right)^{\ord+1}$,
all the above implies that $s_{i}^{\ord}$ has exactly $\ord+1$ real
roots, among them $w=1$ with multiplicity $\ord+1-i$ and all the
rest of the roots being simple and bigger than $1$.
\end{enumerate}
The proof is finished by considering the last item for $i=\ord$.
\end{proof}
Recall \prettyref{prop:pnd-exact-root-z}. Let $\left\{ u_{1}^{\left(\scc\right)}=z,\dots,u_{\ord}^{\left(\scc\right)}\right\} $
denote the roots of $p_{\scc}^{\ord}\left(u\right)$, and $\left\{ w_{1}=1,\dots,w_{\ord}\right\} $
denote the roots of $s_{\ord}^{\ord}\left(w\right)$. 
\begin{prop}
\label{prop:separation}The pairwise distances between $\left\{ u_{1}^{\left(\scc\right)},\dots,u_{\ord}^{\left(\scc\right)}\right\} $
remain $O(1)$ as $\scc\to\infty$.\end{prop}
\begin{proof}
Consider the decomposition \eqref{eq:p-decomposition}. By \prettyref{lem:simple-roots},
$\left\{ w_{1},\dots,w_{\ord}\right\} $ are positive, real and simple
roots of $s_{\ord}^{\ord}\left(w\right)$. By Rouche's theorem, as
$\scc\to\infty$ the roots of $\frac{1}{\scc^{\ord}}p_{\scc}^{\ord}\left(zw\right)$
converge to $\left\{ w_{1},\dots,w_{\ord}\right\} $. Obviously the
polynomials $\frac{1}{\scc^{\ord}}p_{\scc}^{\ord}\left(zw\right)$
and $p_{\scc}^{\ord}\left(zw\right)$ have the same roots, therefore
$\left\{ u_{1}^{\left(\scc\right)},\dots,u_{\ord}^{\left(\scc\right)}\right\} $
also converge to $\left\{ w_{1},\dots,w_{\ord}\right\} $. Since the
pairwise distances between the fixed numbers $\left\{ w_{1},\dots,w_{\ord}\right\} $
do not depend on $\scc$, this finishes the proof.
\end{proof}
Now we can estimate the deviation of the roots of $q_{\scc}^{\ord}$
from $\left\{ u_{1}^{\left(\scc\right)},\dots,u_{\ord}^{\left(\scc\right)}\right\} $. 
\begin{lem}
\label{lem:accuracy-root}Denote by $\left\{ y_{1}^{\left(\scc\right)},\dots,y_{\ord}^{\left(\scc\right)}\right\} $
the roots of $q_{\scc}^{\ord}$, and assume the a-priori bounds of
\prettyref{thm:jump-accuracy-new}. Then there exists $C_{9}=C_{9}\left(\ord\right)$
such that for $\scc\gg1$ and for $j=1,2,\dots,\ord$
\[
\left|y_{j}^{\left(\scc\right)}-u_{j}^{\left(\scc\right)}\right|\leqslant C_{9}\frac{R^{*}}{B^{*}}\scc^{-\ord-1}.
\]
\end{lem}
\begin{proof}
The proof is based on the application of Rouche's theorem. Using the
decomposition \eqref{eq:p-decomposition} and \prettyref{lem:simple-roots},
we have that for $\scc\gg1$ 
\begin{equation}
\left|\frac{\dd}{\dd u}p_{\scc}^{\ord}\left(u\right)\bigl|_{u=u_{i}^{\left(\scc\right)}}\right|\approx\left|\jcc_{\ord}\right|\scc^{\ord},\qquad i=1,2,\dots,\ord.\label{eq:large-derivative}
\end{equation}

In particular, this means that there exists a constant $C_{10}=C_{10}\left(\ord\right)$
such that for all $i=1,2,\dots,\ord$ and $\scc\gg1$
\begin{equation}
\left|\frac{\dd}{\dd u}p_{\scc}^{\ord}\left(u\right)\bigl|{}_{u=u_{i}^{\left(\scc\right)}}\right|\geqslant C_{10}B^{*}\scc^{\ord}.\label{eq:first-der-big}
\end{equation}
Again, from \eqref{eq:p-decomposition} it is easy to see that for
$\scc\gg1$, the high-order derivatives of $p_{\scc}^{\ord}$ at $u_{i}^{\left(\scc\right)}$
can be uniformly bounded by an estimate of the form

\[
\left|\frac{\dd^{k}}{\dd u^{k}}p_{\scc}^{\ord}\left(u\right)\bigl|_{u=u_{i}^{\left(\scc\right)}}\right|\leqslant C_{11}A^{*}\scc^{\ord},\quad k=2,\dots\ord.
\]
for some constant $C_{11}=C_{11}\left(\ord\right)$.

Next we take disks of radius $\eta\left(\scc\right)=C_{9}\frac{R^{*}}{B^{*}}\scc^{-\ord-1}$
around each root $u_{i}^{\left(\scc\right)}$, where $C_{9}$ is to
be determined. Let us fix $1\leqslant i\leqslant\ord$, and consider
the circles
\[
\gamma_{i}^{\left(\scc\right)}=\left\{ t_{\phi}=u_{i}^{\left(\scc\right)}+\eta\left(\scc\right)\ee^{\imath\phi},\;0\leqslant\phi<2\pi\right\} .
\]
By the Taylor formula we have for each $t_{\phi}\in\gamma_{i}^{\left(\scc\right)}$
\begin{eqnarray*}
\left|p_{\scc}^{\ord}\left(t_{\phi}\right)\right| & = & \left|\underbrace{p_{\scc}^{\ord}\left(u_{i}^{\left(\scc\right)}\right)}_{=0}+\underbrace{\frac{\dd}{\dd u}p_{\scc}^{\ord}\left(u_{i}^{\left(\scc\right)}\right)}_{\left|\cdot\right|\geqslant C_{10}B^{*}\scc^{\ord},\;\eqref{eq:first-der-big}}\eta\left(\scc\right)\ee^{\imath\phi}+\frac{1}{2}\frac{\dd^{2}}{\dd u^{2}}p_{\scc}^{\ord}\left(u_{i}^{\left(\scc\right)}\right)\eta^{2}\left(\scc\right)\ee^{2\imath\phi}+\dots\right|\\
_{\left(\scc\gg1\right)} & \geqslant & C_{12}B^{*}\eta\left(\scc\right)\scc^{\ord}=C_{12}C_{9}R^{*}\scc^{-1}\qquad\left(C_{12}=2C_{10}\right).
\end{eqnarray*}

Now consider the\emph{ }perturbation polynomial $e_{\scc}^{\ord}\isdef q_{\scc}^{\ord}-p_{\scc}^{\ord}$.
Its coefficients have magnitudes $\left|\nn{\meas[j\scc]}-\meas[j\scc]\right|\leqslant R^{*}\scc^{-1}$.
Therefore
\[
\left|e_{\scc}^{\ord}\left(t_{\phi}\right)\right|\leqslant C_{13}R^{*}\scc^{-1}.
\]
Note that for $\scc\gg1$ the constant $C_{13}$ does not depend on
$C_{9}$ because, say, $\left|t_{\phi}\right|\leqslant2\left|u_{i}^{\left(\scc\right)}\right|<C^{\sharp}\left(\ord\right)$,
an absolute constant.

Consequently, if we choose $C_{9}=2\frac{C_{13}}{C_{12}}>\frac{C_{13}}{C_{12}}$
we can apply Rouche's theorem and conclude that $q_{\scc}^{\ord}$
has a simple zero within distance $C_{9}\frac{R^{*}}{B^{*}}\scc^{-\ord-1}$
from $u_{i}^{\left(\scc\right)}$.

By \prettyref{prop:separation} the $\left\{ u_{i}^{\left(\scc\right)}\right\} $
are $O\left(1\right)$-separated, therefore if $\scc$ is large enough
then the quantity $C_{9}\frac{R^{*}}{B^{*}}\scc^{-\ord-1}$ will be
smaller than the minimal separation distance.\end{proof}
\begin{rem}
\label{rem:all-roots-are-good}This analysis is valid for \emph{any
root} of $q_{\scc}^{\ord}$, not just the perturbation of $u_{1}=z$.
The roots of $q_{\scc}^{\ord}$ all lie approximately on the ray with
angle $\jp\scc$. This means that the parameter $\jp$ can be recovered
with high accuracy from any root of $q_{\scc}^{\ord}\left(u\right)$,
and we expect that it might be important for practice (so for instance
one can approximate $\jp$ by averaging).\end{rem}
\begin{proof}
[Proof of \prettyref{thm:jump-accuracy-new}, first part.]Let us
track steps 2-4 of \prettyref{alg:new-reconstruction-single-jump}.
\begin{itemize}
\item By \prettyref{lem:accuracy-root}, the accuracy of step 2 is bounded
by $C_{9}\frac{R^{*}}{B^{*}}\scc^{-\ord-1}$, i.e. we can write
\[
\nn z_{\scc}=z+\frac{C^{*}\left(\scc\right)}{\scc^{\ord+1}}
\]
where $\left|C^{*}\left(\scc\right)\right|\leqslant C_{9}\frac{R^{*}}{B^{*}}$.
\item Extraction of $\scc^{\text{th}}$ root in step 3 further decreases
the error by the factor $\frac{1}{\scc}$. Indeed, we have
\begin{eqnarray*}
\left|\nn{\w}_{\scc}-\w\right| & = & \left|1-\frac{\nn{\w}_{\scc}}{\w}\right|=\left|1-\left(\frac{\nn z_{\scc}}{z}\right)^{\frac{1}{\scc}}\right|\\
_{\left(\left|C^{**}\left(\scc\right)\right|\leqslant C_{9}\frac{R^{*}}{B^{*}}\right)} & = & \left|1-\left(1+\frac{C^{**}\left(\scc\right)}{\scc^{\ord+1}}\right)^{\frac{1}{\scc}}\right|\\
_{\left(\text{Bernoulli's inequality}\right)} & \leqslant & C_{9}\frac{R^{*}}{B^{*}}\scc^{-\ord-2}.
\end{eqnarray*}

\item Step 4 preserves this estimate, since 
\begin{eqnarray*}
\nn{\w}_{\scc} & = & \w+C_{\$}\left(\scc\right)\frac{R^{*}}{B^{*}}\scc^{-\ord-2},\qquad\left|C_{\$}\left(\scc\right)\right|\leqslant C_{9}\\
\Longrightarrow\left|\nn{\jp}_{\scc}-\jp\right| & = & \left|\log\frac{\nn{\w}_{\scc}}{\w}\right|=\left|\log\left(1+C_{\$\$}\left(\scc\right)\frac{R^{*}}{B^{*}}\scc^{-\ord-2}\right)\right|\qquad\left(\left|C_{\$\$}\left(\scc\right)\right|\leqslant C_{9}\right)\\
 & \leqslant & 2C_{9}\frac{R^{*}}{B^{*}}\scc^{-\ord-2},
\end{eqnarray*}
the last inequality following from the estimate $\left|\log\left(1+\e\right)\right|<2\left|\e\right|$
for $\left|\e\right|\ll1$.
\end{itemize}

The proof of the first part is therefore finished with $C_{4}\isdef2C_{9}$.

\end{proof}
\begin{minipage}[t]{1\columnwidth}%
\end{minipage}
\begin{proof}
[Proof of \prettyref{thm:jump-accuracy-new}, second part.] We have
recovered the approximate value $\widetilde{\w}_{\scc}$ which satisfies
$\left|\widetilde{\w}_{\scc}-\w\right|\leq C_{4}\frac{R^{*}}{B^{*}}\scc^{-\ord-2}$,
while $\left|\apprmeas-\meas\right|\leq R^{*}k^{-1}$. Now we estimate
the corresponding error in the solution of the linear system \eqref{eq:linear-system-perturbed}.

By \eqref{eq:the-linear-system} and \eqref{eq:linear-system-perturbed},
the error vector satisfies
\[
\begin{bmatrix}\widetilde{\jcc}_{0}-\jcc_{0}\\
\widetilde{\jcc}_{1}-\jcc_{1}\\
\vdots\\
\widetilde{\jcc}_{\ord}-\jcc_{\ord}
\end{bmatrix}=\left(V_{\scc}^{\ord}\right)^{-1}\begin{bmatrix}\apprmeas[\scc]\widetilde{\w}_{\scc}^{-\scc}-\meas[\scc]\w^{-\scc}\\
\apprmeas[2\scc]\widetilde{\w}_{\scc}^{-2\scc}-\meas[2\scc]\w^{-2\scc}\\
\vdots\\
\apprmeas[\left(\ord+1\right)\scc]\widetilde{\w}_{\scc}^{-\left(\ord+1\right)\scc}-\meas[\left(\ord+1\right)\scc]\w^{-\left(\ord+1\right)\scc}
\end{bmatrix}.
\]
Since
\begin{eqnarray*}
\left|\meas[j\scc]\right| & \leq & C_{14}A^{*}\scc^{\ord}\\
\nn{\meas[j\scc]} & = & \meas[j\scc]+R\left(\scc\right)\scc^{-1},\qquad R\left(\scc\right)\leqslant R^{*}\\
\nn{\w}_{\scc} & = & \w+C_{\$}\left(\scc\right)\frac{R^{*}}{B^{*}}\scc^{-\ord-2},\qquad C_{\$}\left(\scc\right)\leqslant C_{4}
\end{eqnarray*}
then we have (using standard Taylor majorization techniques, see e.g.
\cite[Proposition A.7]{batyomAlgFourier}) that
\begin{eqnarray*}
\nn{\w}_{\scc}^{-j\scc} & = & \left(\w+C_{4}\left(\scc\right)\frac{R^{*}}{B^{*}}\scc^{-\ord-2}\right)^{-j\scc}\\
 & = & \w^{-j\scc}\left(1+\frac{C_{4}\left(\scc\right)R^{*}}{B^{*}\w}\scc^{-\ord-2}\right)^{-j\scc}\\
 &  & \w^{-j\scc}\left(1-C_{15}\left(\scc\right)\frac{R^{*}}{B^{*}}\scc^{-\ord-1}\right)
\end{eqnarray*}
with $\left|C_{15}\left(\scc\right)\right|\leqslant2C_{9}$, and consequently
\begin{eqnarray}
\begin{split}\left|\apprmeas[j\scc]\widetilde{\w}_{\scc}^{-j\scc}-\meas[j\scc]\w^{-j\scc}\right| & =\left|\left(\meas[j\scc]+R\left(\scc\right)\scc^{-1}\right)\w^{-j\scc}\left(1-C_{15}\left(\scc\right)\frac{R^{*}}{B^{*}}\scc^{-\ord-1}\right)-\meas[j\scc]\w^{-j\scc}\right|\\
 & \leqslant\frac{R^{*}}{\scc}\left|\frac{2C_{9}C_{14}A^{*}}{B^{*}}+1\right|+O\left(\scc^{-\ord-2}\right)\\
 & \leq C_{16}R\left(1+\frac{A^{*}}{B^{*}}\right)\scc^{-1},\qquad C_{16}=\max\left\{ 1,2C_{9}C_{14}\right\} .
\end{split}
\label{eq:rhs-error-linear}
\end{eqnarray}

Denote $\zeta_{j}=\apprmeas[j\scc]\widetilde{\w}_{\scc}^{-j\scc}-\meas[j\scc]\w^{-j\scc}$
and also let $C_{17,\ell}$ be an upper bound on the sum of absolute
values of the entries in the $\ell$-th row of $\left(V_{1}^{\ord}\right)^{-1}$.
It is immediate that
\[
V_{\scc}^{\ord}=V_{1}^{\ord}\diag\left\{ 1,\scc,\dots,\scc^{\ord}\right\} ,
\]
therefore
\[
\begin{bmatrix}\widetilde{\jcc}_{0}-\jcc_{0}\\
\widetilde{\jcc}_{1}-\jcc_{1}\\
\vdots\\
\widetilde{\jcc}_{\ord}-\jcc_{\ord}
\end{bmatrix}=\begin{bmatrix}1\\
 & \scc^{-1}\\
 &  & \ddots\\
 &  &  & \scc^{-\ord}
\end{bmatrix}\left(V_{1}^{\ord}\right)^{-1}\begin{bmatrix}\zeta_{0}\\
\zeta_{1}\\
\vdots\\
\zeta_{\ord}
\end{bmatrix}.
\]
Using the estimate \eqref{eq:rhs-error-linear} we have immediately
that for $\ell=0,1,\dots,\ord$
\[
\left|\widetilde{\jcc}_{\ell}-\jcc_{\ell}\right|\leqslant C_{5,\ell}R^{*}\left(1+\frac{A^{*}}{B^{*}}\right)\scc^{-\ell-1},
\]
where $C_{5,\ell}\isdef C_{16}C_{17,\ell}$. This completes the proof
of \prettyref{thm:jump-accuracy-new}.
\end{proof}

\section{\label{sec:numerics}Numerical experiments}

In our numerical experiments we compared the performance of the following
Eckhoff-based methods for recovery of a single jump point position
from the first $\sc$ Fourier coefficients: original Eckhoff's formulation
from \cite{eckhoff1995arf} (\noun{Eckhoff}); our previous method
from \cite{batyomAlgFourier} (\noun{BY 2011}); the method presented
in this paper (\noun{Full}). All the three methods in essense solve
a polynomial equation $p_{\sc}\left(u\right)=0$ satisfied by the
jump point $\w=\ee^{-\imath\jp}$: for \noun{Eckhoff} and \noun{BY 2011}
this polynomial is constructed from consecutive samples $k=\sc-\ord-1,\dots,\sc$
(\prettyref{alg:half-algo}), while \noun{Full} uses the decimated
sequence $k=\scc,2\scc,\dots,,\left(\ord+2\right)\scc$. The only
difference between \noun{Eckhoff} and \noun{BY 2011} is the degree
of $p_{\sc}\left(u\right)$: the former uses the full smoothness $\ord$
while the latter uses $\ord_{1}=\left\lfloor \frac{\ord}{2}\right\rfloor $.
The jump point $\jp$ and the magnitudes, as well as the error terms,
are randomly chosen in the beginning of the whole experiment.

All calculations were done using Mathematica software with high-precision
setting. The results, presented in \prettyref{fig:experiments}, agree
well with the theory - \noun{Full} presents an improvement of $\sim\sc^{-\ord/2}$
compared to \noun{BY 2011}, and improvement of order $\sc^{-\ord}$
compared to \noun{Eckhoff}.

\begin{figure}
\subfloat[Order=3]{\includegraphics[angle=90,width=0.4\linewidth]{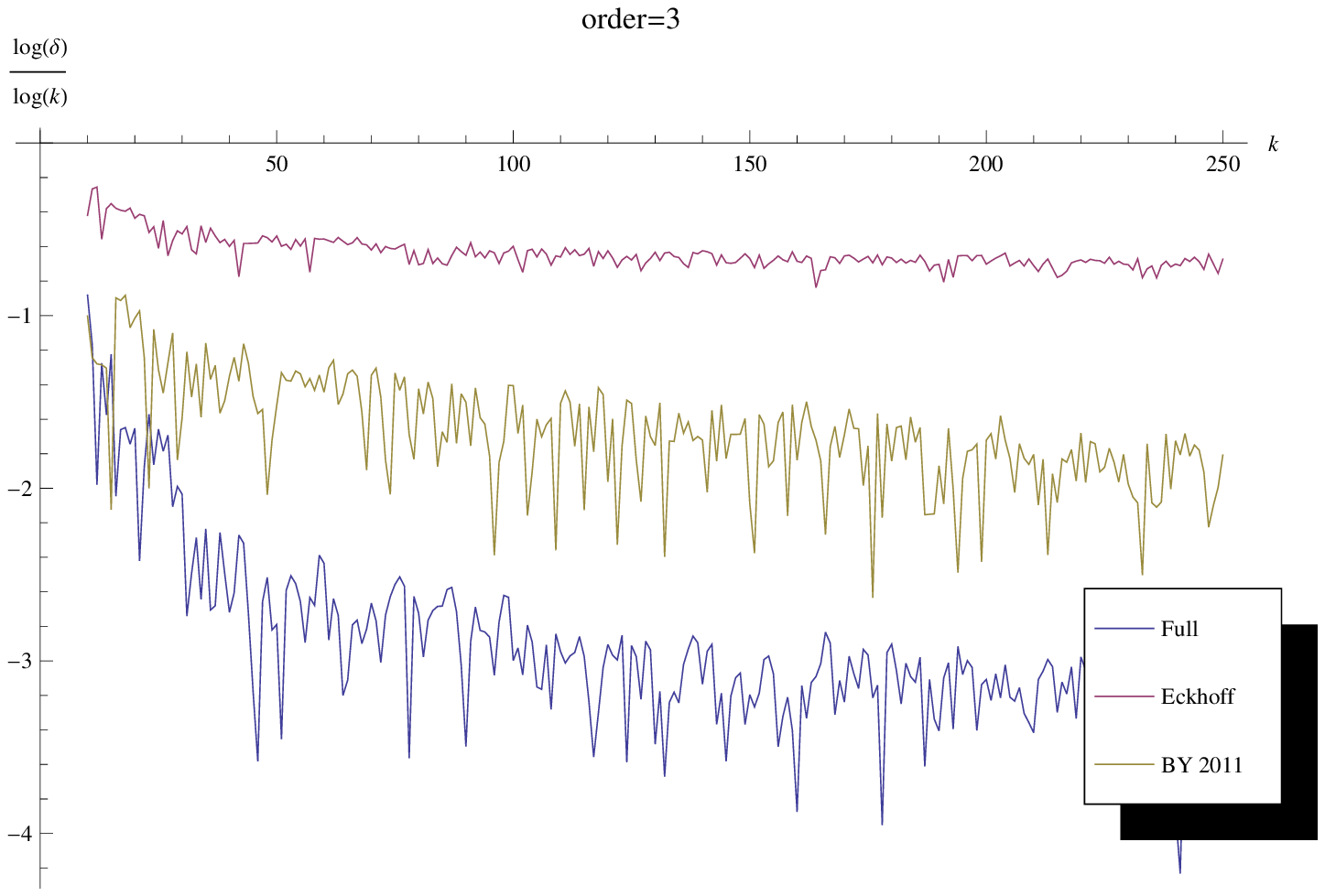}

}\hspace{1cm}\subfloat[Order=4]{\includegraphics[angle=90,width=0.4\linewidth]{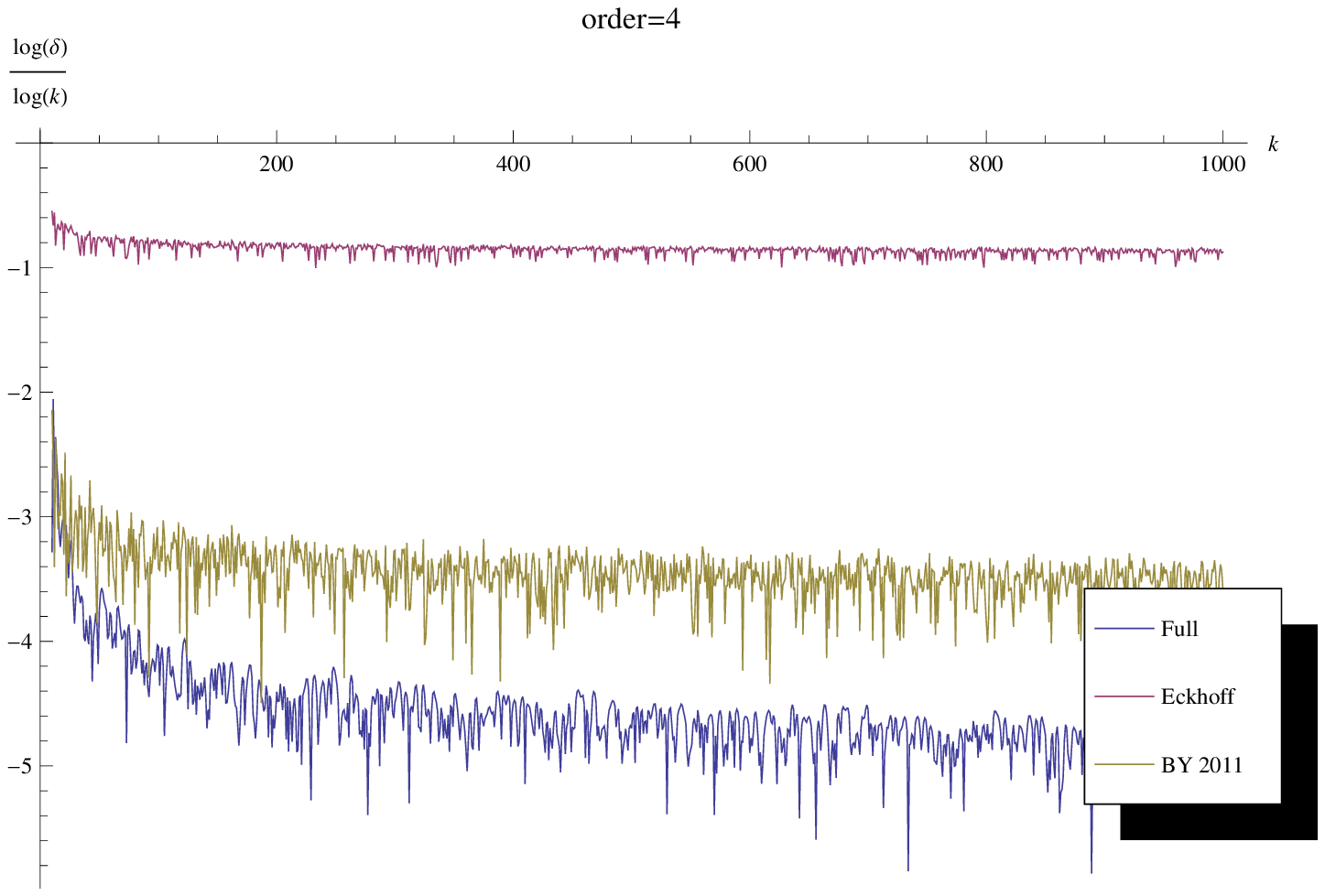}

}

\subfloat[Order=5]{\includegraphics[clip,angle=90,width=0.4\linewidth]{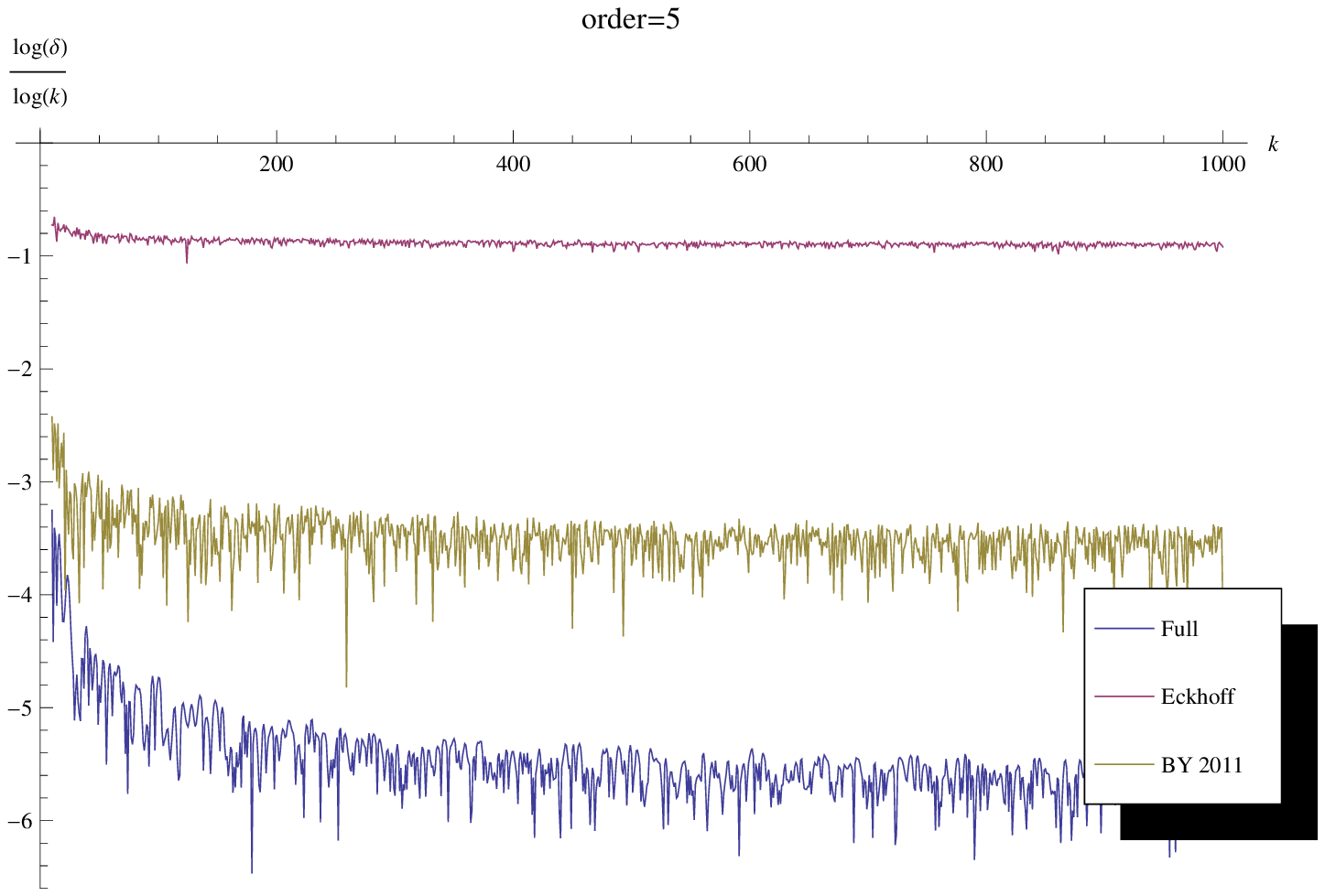}

}\hspace{1cm}\subfloat[Order=7]{\includegraphics[clip,angle=90,width=0.4\linewidth]{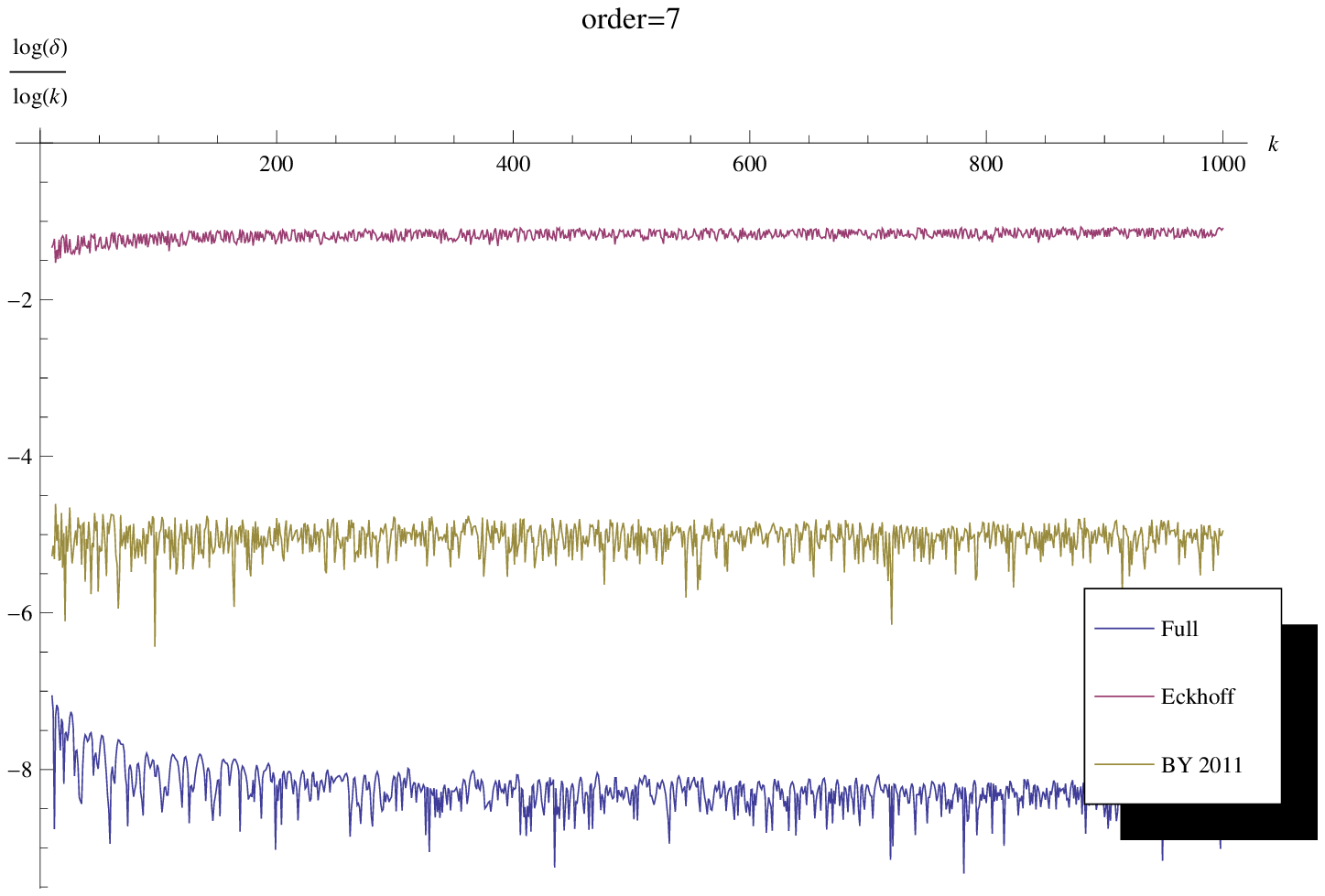}

}

\caption{{\footnotesize{\protect\noun{{\footnotesize{Full}}} represents the
algorithm of this paper, \protect\noun{{\footnotesize{BY 2011}}}
referes to the method of \cite{batyomAlgFourier} while \protect\noun{{\footnotesize{Eckhoff}}}
denotes the original method of Eckhoff from \cite{eckhoff1995arf}.
The $x$ axis shows the index $k$ used for the reconstruction, corresponding
to the number $\sc$ in the text. The $y$ axis shows the ratio $\frac{\log\delta}{\log k}$,
where $\delta$ is the reconstruction error exhibited by the algorithms.}}}
\label{fig:experiments}
\end{figure}

\section{\label{sec:optimality}Practical aspects of algebraic reconstruction}

\global\long\def\pm{\mathcal{P}}
\global\long\def\jac{\mathcal{J}}
\global\long\def\ordd{\tilde{\ord}}

\subsection{Stability of the algorithm and Prony-type systems}

The optimality, or efficiency, of the proposed algorithm remains an
important practical issue. It is immediately seen that our algorithm
is stable with respect to perturbations in the Fourier coefficients
$\fc\left(\fun\right)$ \emph{of the order $O\left(k^{-\ord-2}\right)$
}(since such perturbations will just be absorbed into the constant
$R$ appearing in \eqref{eq:decimated-system-single-jump}). This
means, however, that the higher coefficients need to be acquired with
increasing accuracy, which might very well be impossible in practice.
While best possible asymptotic rate of convergence is achieved, it
comes at the cost of high-precision computations and a large number
of required Fourier coefficients (see e.g. experiments on localization
procedure in \cite{batyomAlgFourier} where convergence starts with
large $\sc$). So in terms of \emph{actual performance,} the ``decimated
Eckhoff algorithm'' is probably not the best currently available
method for jump detection\emph{ }in real-world scenarios\emph{. }For
this reason, at this stage we do not attempt to compare its performance
to well-known methods such as concentration kernels. Instead, in this
section we briefly discuss the question of best \emph{absolute} performance
of \emph{any method whatsoever.}

Consider the Eckhoff's problem without reference to any concrete method.
A formulation which might be more suitable for practical applications
is the following.
\begin{problem}
\label{prob:best-abs-error}Given first $\sc$ Fourier coefficients
of $\fun\in PC\left(\ord+1,\np\right)$, possibly with some perturbations
bounded by $\leqslant\delta$, find the points of discontinuity of
$\fun$ with \emph{smallest absolute error}.
\end{problem}
The problem is that, as far as we are aware, even the question of
determining \emph{what the smallest absolute error actually is}, remains
open. Motivated by this question, we have started investigating the
so-called ``Prony type'' systems%
\footnote{These systems are important in many problems of mathematics and engineering
\cite{Auton1981}. They have been used as far back as by Baron de
Prony in 1795 \cite{prony1795essai} for the problem of exponential
fitting.%
} (of which \eqref{eq:eckhoff-system} is a special case), in particular
lower bounds for their solution. Let us now briefly discuss the relevant
results of \cite{batenkov2011accuracy,batPronyDec} in the context
of Eckhoff's problem.

Consider the following ``polynomial Prony'' system of equations:

\begin{equation}
\meas=\sum_{j=1}^{\np}z_{j}^{k}\sum_{\ell=0}^{\ell_{j}-1}\jc_{\ell,j}k^{\ell},\quad\left|z_{j}\right|=1,\; a_{\ell,j}\in\complexfield,\;\sum_{j=1}^{\np}\ell_{j}=C.\label{eq:polynomial-prony}
\end{equation}
Denote the overall number of unknown by $R\isdef C+\np$. Assume that
we are given the measurement sequence $\left\{ \meas\right\} _{k=0,1,\dots,\sc}$.
Choose an index set $S\subset\left\{ 0,1,\dots,\sc\right\} $ of size
exactly $R$. This defines the so-called ``Prony map'' $\pm:\complexfield^{R}\to\complexfield^{R}$,
which maps the parameters $\left\{ z_{j},a_{\ell,j}\right\} $ to
the measurements $\left\{ \meas\right\} _{k\in S}$. This also defines
the ``reconstruction map'' $\pm^{-1}$, which can be thought of
as representing an ``ideal reconstruction algorithm''. In a small
neighborhood of a regular (i.e. non-critical) point of $\pm$, the
map $\pm^{-1}$ is well-defined and well-approximated by its linear
part, given by the Jacobian matrix $\jac$. Consequently, if the left-hand
side of \eqref{eq:polynomial-prony} is perturbed by a small amount
$\varepsilon\ll1$, then the corresponding perturbation in the values
of $\left\{ z_{j},a_{\ell,j}\right\} $ can be easily bounded by the
sum of the magnitudes of the entries of the corresponing row of $\jac$
times $\varepsilon$.

Let the set $S$ be of the form of an arithmetic progression with
initial value $t$ and step size $\sigma$, i.e.
\begin{equation}
S=\left\{ t,t+\sigma,t+2\sigma,\dots,t+\left(R-1\right)\sigma\right\} .\label{eq:sampling-set-progr}
\end{equation}

Under the above assumptions, in \cite{batenkov2011accuracy,batPronyDec}
we have shown that the error for recovering the jump $z_{j}$ satisfies
\begin{eqnarray}
\left|\Delta z_{j}\right| & \leqslant & \frac{2}{\ell_{j}!}\left(\frac{2}{\delta_{\sigma}}\right)^{R}\frac{1}{\left|\jc_{\ell_{j}-1,j}\right|\sigma^{\ell_{j}}}\varepsilon,\label{eq:poly-node-pert}
\end{eqnarray}
where $\delta_{\sigma}\isdef\min_{i\neq j}\left|z_{i}^{\sigma}-z_{j}^{\sigma}\right|$.
A similar, slightly more involved expression is provided for $\left|\Delta\jc_{\ell,j}\right|$.

Now consider the system \eqref{eq:eckhoff-system}. Multiplying both
sides by $2\pi\left(\imath k\right)^{\ord+1}$, we obtain the system
\eqref{eq:polynomial-prony} with $\ell_{j}=\ord+1$ and perturbation
of size $\varepsilon=O\left(\sc^{-1}\right)$.

Take $S_{1}=\left\{ \sc-\left(\ord+2\right)\np+1,\dots,\sc\right\} $,
which corresponds to the original Eckhoff method of \cite{eckhoff1995arf}.
By \eqref{eq:poly-node-pert} we get $\left|\Delta z_{j}\right|=O\left(\sc^{-1}\right)$,
i.e. only first order accuracy. In contrast, for
\[
S_{2}=\left\{ \left\lfloor \frac{\sc}{\left(\ord+2\right)\np}\right\rfloor ,2\left\lfloor \frac{\sc}{\left(\ord+2\right)\np}\right\rfloor ,\dots,\left(\ord+2\right)\np\left\lfloor \frac{\sc}{\left(\ord+2\right)\np}\right\rfloor \right\} 
\]
we get $\left|\Delta z_{j}\right|=O\left(\sc^{-\left(\ord+2\right)}\right)$,
i.e. maximal possible asymptotic accuracy. Thus, the Prony systems
approach provides another justification for the decimation technique.

But it can provide much more. Indeed, the magnitude of the norm of
the Jacobian (bounded from above by \eqref{eq:poly-node-pert}) provides
by definition \emph{the best possible} \emph{stability bounds} (at
least in the case of small perturbations), and therefore the performance
(including robustness to noise) of all algorithms (strictly speaking,
of those which utilize sampling sets of the form \eqref{eq:sampling-set-progr})
should be compared to these bounds.

To demonstrate this point, consider the decimated Eckhoff algorithm
for one point, i.e. \prettyref{alg:new-reconstruction-single-jump},
for the system \eqref{eq:noisy-measurements}, and its stability as
provided by \prettyref{thm:jump-accuracy-new}. Application of the
bound \eqref{eq:poly-node-pert} to this case gives (here $\delta_{\sigma}$
is effectively equal to 1, and also $R=\left(\ord+2\right),\;\left|\jc_{\ord}\right|>B^{*},\;\left|\varepsilon\right|\leqslant R^{*}\cdot\sc^{-1}$
and $\sigma=\scc=\left\lfloor \frac{\sc}{\left(\ord+2\right)}\right\rfloor $)
\begin{eqnarray}
\left|\Delta z_{j}\right| & \leqslant & \frac{2^{\ord+2}}{\left(\ord+1\right)!}\cdot\frac{1}{\left|\jc_{\ord}\right|\scc^{\ord+1}}\varepsilon<\frac{2^{\ord+2}\left(\ord+2\right)}{\left(\ord+1\right)!}\cdot\frac{R^{*}}{B^{*}}\cdot\scc^{-\ord-2}.\label{eq:deltazj-prony}
\end{eqnarray}

On the other hand, according to the proof of \prettyref{thm:jump-accuracy-new},
we have
\[
\left|\Delta\w\right|\leqslant C_{9}\left(\ord\right)\frac{R^{*}}{B^{*}}\scc^{-\ord-2}.
\]
Thus, it can be said that \prettyref{alg:new-reconstruction-single-jump}
provides qualitatively best performance, as both estimates are proportional
to $\frac{R^{*}}{B^{*}}$. The following calculation provides a simple
estimate of the constant $C_{9}$.
\begin{prop}
\label{prop:c9-refinement}If in step 2 of \prettyref{alg:new-reconstruction-single-jump}
the closest root to the unit circle is chosen, then the constant $C_{9}$
satisfies
\begin{eqnarray}
C_{9} & \leqslant & \frac{3^{\ord+1}}{\left(\ord+1\right)!}.\label{eq:c9-refined}
\end{eqnarray}
\end{prop}
\begin{proof}
Using the fact that $w=1$ is a multiple root of $s_{i}^{\ord}$ for
$i<\ord$ and the decomposition \eqref{eq:p-decomposition}, we obtain
that
\[
\frac{\dd}{\dd u}p_{\scc}^{\ord}\left(u\right)\bigg|_{u=z}=\left(\ord+1\right)!\jcc_{\ord}\scc^{\ord},
\]
and therefore in \eqref{eq:first-der-big} we can take $C_{10}=\left(\ord+1\right)!$.
Thus, $C_{12}=2C_{10}=2\left(\ord+1\right)!$ . To estimate $C_{13}$,
we further have for $\left|t_{\phi}\right|<2$
\[
\left|e_{\scc}^{\ord}\left(t_{\phi}\right)\right|\leqslant R^{*}\scc^{-1}\underbrace{\sum_{j=0}^{\ord+1}{\ord+1 \choose j}2^{\ord+1-j}}_{=3^{\ord+1}},
\]
and thus $C_{13}=3^{\ord+1}$. Finally, $C_{9}=\frac{2C_{13}}{C_{12}}=\frac{2\cdot3^{\ord+1}}{2\left(\ord+1\right)!}$,
which proves \eqref{eq:c9-refined}.
\end{proof}
The formula \eqref{eq:poly-node-pert} turns out to be fairly tight,
and thus by comparing \eqref{eq:deltazj-prony} with \eqref{eq:c9-refined}
it can be said that \prettyref{alg:new-reconstruction-single-jump}
is away from best accuracy by a factor of
\[
\frac{\left(\frac{3}{2}\right)^{\ord+1}}{2\left(\ord+2\right)}.
\]

Similar calculations can be performed for the perturbations in the
magnitudes, but due to more complicated expressions we do not present
them here.

In order to obtain absolute error bounds for \prettyref{prob:best-abs-error}
(and for instance compare them with the constants in \prettyref{thm:full-accuracy-final}),
the above approach should be extended to handle neighborhoods of finite
size, as well as the overdetermined setting (i.e. the case $\left|S\right|>R$).
We consider this to be an important question for future investigation.

\subsection{Incorrect choice of the smoothness parameter}

An important feature of our method is that the parameters $\ord,\np$
are assumed to be known a-priori. Even in the case of one jump, an
oversetimation of the order $\ord$ leads to the overall deterioration
of the accuracy%
\footnote{In contrast, underestimation might lead to cancellation effects such
as the one described in \cite{batyomAlgFourier}.%
}. Let us briefly show this.

Assume that the function $\fun_{j}$ is only piecewise $\ordd$-smooth,
i.e. $\fun_{j}\in PC\left(\ordd,1\right)$, when $\ordd<\ord$, but
\prettyref{alg:new-reconstruction-single-jump} is applied with order
$\ord$.  The formula \eqref{eq:p-decomposition} would now read
\[
p_{\scc}^{\ord}\left(zw\right)=z^{\ord+2}\sum_{i=0}^{\ordd}\jcc_{i}\scc^{i}s_{i}^{\ord}\left(w\right).
\]
Consequently, in the perturbation analysis of \prettyref{lem:accuracy-root},
we would have that in a small $\varepsilon$-neighborhood of $z=\w^{\scc}$,
the polynomial $p_{\scc}^{\ord}$ is approximately of magnitude $O\left(\scc^{\ordd}\right)\varepsilon$.
On the other hand, the term $\epsilon_{k}$ in \eqref{eq:decimated-system-single-jump}
is of the order $O\left(k^{-\ordd-2}\right)$, and subsequently the
term $\delta_{k}$ in \eqref{eq:noisy-measurements} is of order $O\left(\scc^{\ord-\ordd-1}\right)$.
Therefore, the polynomial $e_{\scc}^{\ord}$ has coefficients of the
order $\left|\nn{\meas[j\scc]}-\meas[j\scc]\right|=O\left(\scc^{\ord-\ordd-1}\right)$.
Consequently, the size of the $\varepsilon$-neighborhood containing
the perturbed root of $q_{\scc}^{\ord}$ is in general not better
than $\varepsilon=O\left(\scc^{\ord-2\ordd-1}\right)$. To conclude,
in this case the jump point would be detected with accuracy $O\left(\scc^{\ord-2\ordd-2}\right)$
which is of course worse than $O\left(\scc^{-\ordd-2}\right)$ (the
best possible for piecewise $\ordd$-smooth case).

In the general setting of Prony systems (and in Eckhoff's problem
in particular), the problem of estimating the model parameters $\np,\left\{ \ell_{j}\right\} $
from the Fourier data appears to be challenging, especially in the
presence of closely spaced jumps and noise. Recent studies (such as
\cite{candes2012towards}) suggest that in any such setting, a crucial
role is played by the a-priori minimal node separation assumption.
On the other hand, the overall degree $\sum\ell_{j}$ of the Prony
system \eqref{eq:polynomial-prony} can be estimated via the numerical
rank of certain Hankel matrices constructed from the data $\left\{ \meas\right\} $
(see e.g. \cite{rao1992mbp} and references therein), and this information,
combined with the node separation assumption, might be used for the
correct ``clustering''. The basis of divided differences might also
play an important role in this problem, see \cite{byPronySing12,yom2009Singularities}.

\section{\label{sec:extensions}Possible extensions}
\begin{enumerate}
\item The Eckhoff's method has been extended in the literature to handle
expansions in other orthonormal basis, such as Chebyshev series (\cite{eckhoff1993,eckhoff1995arf})
and Fourier-Jacobi series (\cite{kvernadze2004ajd}). It should be
fairly straightforward to extend \prettyref{alg:full-algo} and the
analysis of \prettyref{sec:accuracy} to handle these cases.
\item Another immediate generalization is to the case of piecewise $C^{\infty}$
functions. By increasing the order $\ord$ of the reconstruction,
according to \prettyref{thm:full-accuracy-final} the resulting accuracy
will eventually be asymptotically smaller than any algebraic power
of $\sc$. This comes, however, at the cost of the constants of proportionality
growing with $\ord$.
\item This last remark brings us to another possible generalization, namely
to reconstruction of piecewise-\emph{analytic} functions. One natural
line of attack would be to analyze how the constants appearing in
the accuracy estimates depend on the smoothness order $\ord$ (as
in the special case provided by \prettyref{prop:c9-refinement}),
and then choose $\ord$ in an appropriate way so as to maximize the
resulting accuracy ($\ord$ would be depending on $\sc$ in this case).
According to the results of \cite{adcock2012stability}, one may expect
(at most) stable root-exponential convergence and unstable exponential
convergence. We plan to develop these ideas in a future work.
\item As noted by K.Eckhoff in \cite{eckhoff1995arf}, the methods can easily
be adjusted to handle discontinuities in higher derivatives (and not
in the function itself). We expect that decimation will provide the
best asymptotic convergence also in these cases.
\item Extension of the one-dimensional algebraic methods to higher dimensions
seems to be highly nontrivial, but nevertheless possible for some
special geometric configurations \cite{BatGolYom2011,eckhoff1998high}.
We consider it to be an important topic for future investigations.
\end{enumerate}
\bibliographystyle{plain}
\bibliography{../../../bibliography/all-bib}

\appendix

\section{\label{app:max-accuracy}Maximal accuracy for jumps}
\begin{proof}
[Proof of \prettyref{prop:maximal-accuracy}]Consider the following
subset of $PC\left(\ord+1,\np\right)$
\[
B\left(A,R\right)=\left\{ \fun\in PC\left(\ord+1,\np\right):\;\fun=\sing^{\left(\ord\right)}+\smooth;\;\left|\fc\left(\smooth\right)\right|<R\cdot k^{-\ord-2};\;\sum_{\ell,j}\left|\jc_{\ell,j}\right|<A\right\} 
\]
where the smooth part $\smooth$ is in $C^{\ord+1}$ and the quantities
$\left\{ \jc_{\ell.j}\right\} $ denote the associated jump magnitudes
of the piecewise polynomial $\sing^{\left(\ord\right)}$ of degree
$\ord$, as in \eqref{eq:sing-part-explicit-bernoulli}.

Let $g\in B\left(A,R\right)$ be an arbitrary fixed piecewise polynomial
$g=\sing^{\left(\ord\right)}$ with jumps $\left\{ \jp_{1},\dots,\jp_{\np}\right\} $
and associated jump magnitudes $\left\{ \jc_{\ell,j}\right\} $ .
We will show that there exists an absolute constant $C$ such that
for every index $\sc$ there exists a function $h_{\sc}\in B\left(A,R\right)$
whose first $\sc$ Fourier coefficients coincide with those of $g$,
while the corresponding jump locations differ by $C\sc^{-\ord-2}$.
Once we show this, it is clear that no deterministic algorithm will
be able to reconstruct the jump locations of all functions in $B\left(A,R\right)$
with accuracy essentially better than $O\left(\sc^{-\ord-2}\right)$.

Denote $\delta=C\sc^{-\ord-2}$ where $C$ is to be determined. Let
$\sing_{\sc}^{\left(\ord\right)}$ denote another piecewise polynomial
of degree $\ord$ with jumps
\[
\left\{ \eta_{1}=\jp_{1}+\delta,\dots,\eta_{\np}=\jp_{\np}+\delta\right\} 
\]
and the same jump magnitudes $\left\{ \jc_{\ell,j}\right\} $ as those
of $\sing^{\left(\ord\right)}$. Let
\[
b_{k}=\fc\left(\sing^{\left(\ord\right)}-\sing_{\sc}^{\left(\ord\right)}\right).
\]
Finally take
\[
h_{\sc}\left(x\right)\isdef\sing_{\sc}^{\left(\ord\right)}\left(x\right)+\sum_{\left|k\right|=0}^{\sc}b_{k}\ee^{\imath kx}.
\]
Clearly $\fc\left(g\right)=\fc\left(h_{\sc}\right)$ for $\left|k\right|=0,1,\dots,\sc$.
In order to ensure that $h_{\sc}\in B\left(A,R\right)$ we need to
choose $C$ small enough such that 
\[
\left|b_{k}\right|\leq R\cdot k^{-\ord-2};\qquad k=1,2,\dots,\sc.
\]
Let us show that $C\isdef\frac{2\pi R}{A}$ satisfies the above condition.
Indeed:
\begin{eqnarray*}
b_{k}=\fc\left(\sing^{\left(\ord\right)}-\sing_{\sc}^{\left(\ord\right)}\right) & = & \frac{1}{2\pi}\sum_{j=1}^{\np}\ee^{-\imath\jp_{j}k}\sum_{\ell=0}^{\ord}\frac{\jc_{\ell,j}}{\left(\imath k\right)^{\ell+1}}-\frac{1}{2\pi}\sum_{j=1}^{\np}\ee^{-\imath\eta_{j}k}\sum_{\ell=0}^{\ord}\frac{\jc_{\ell,j}}{\left(\imath k\right)^{\ell+1}}\\
 & = & \left(\ee^{\imath\delta k}-1\right)\frac{1}{2\pi}\sum_{j=1}^{\np}\ee^{-\imath\eta_{j}k}\sum_{\ell=0}^{\ord}\frac{\jc_{\ell,j}}{\left(\imath k\right)^{\ell+1}}.
\end{eqnarray*}
Now obviously
\[
\left|\frac{1}{2\pi}\sum_{j=1}^{\np}\ee^{-\imath\eta_{j}k}\sum_{\ell=0}^{\ord}\frac{\jc_{\ell,j}}{\left(\imath k\right)^{\ell+1}}\right|\leq\frac{A}{2\pi k}.
\]
From geometric considerations we have $\left|1-\ee^{\imath\delta k}\right|\leq\delta k$,
therefore
\begin{eqnarray*}
\left|b_{k}\right| & \leq & \frac{A}{2\pi k}\delta k=\frac{A}{2\pi}C\sc^{-\ord-2}=\frac{A}{2\pi}\cdot\frac{2\pi R}{A}\cdot\sc^{-\ord-2}<R\sc^{-\ord-2}.
\end{eqnarray*}
This completes the proof.
\end{proof}

\end{document}